\documentclass{amsart}
\usepackage{graphicx}
\usepackage{amsmath,amssymb}

\newtheorem{thm}{Theorem}[section]
\newtheorem{cor}[thm]{Corollary}
\newtheorem{lem}[thm]{Lemma}
\newtheorem{prop}[thm]{Proposition}

\theoremstyle{definition}

\theoremstyle{remark}

\theoremstyle{example}

\numberwithin{equation}{section}

\newcommand{\eps}{\varepsilon}

\newcommand{\sig}{\sigma}

\newcommand{\del}{\delta}

\newcommand{\B}{{\mathbb B}}
\newcommand{\D}{{\mathbb D}}
\newcommand{\R}{{\mathbb R}}
\newcommand{\SSS}{{\mathbb S}}
\newcommand{\C}{{\mathbb C}}
\newcommand{\Q}{{\mathbb Q}}
\newcommand{\N}{{\mathbb N}}

\newcommand{\conj}[1]{\overline{#1}}

\newcommand{\calB}{{\mathcal B}}

\newcommand{\calM}{{\mathcal M }}
\newcommand{\calN}{{\mathcal N}}
\newcommand{\calO}{{\mathcal O}}
\newcommand{\calQ}{{\mathcal Q}}

\newcommand{\Np}{\mathcal{N}_{p}}
\DeclareMathOperator{\Aut}{Aut}
\DeclareMathOperator{\Mult}{Mult}

\allowdisplaybreaks[1]

\begin{document}

\title{On the structure of $\calN_p$-spaces in the ball}%

\date{\today}%

\author{Bingyang Hu$^\dag$, Le Hai Khoi$^\dag$ and Trieu Le}%

\address{(Hu) Department of Mathematics, University of Wisconsin, Madison, WI 53706, USA}%
\email{BingyangHu@math.wisc.edu}

\address{(Khoi) Division of Mathematical Sciences, School of Physical and Mathematical Sciences, Nanyang Technological University (NTU),
637371 Singapore}%
\email{lhkhoi@ntu.edu.sg}

\address{(Le) Department of Mathematics and Statistics, Mail Stop 942, University of Toledo, Toledo, OH 43606, USA}%
\email{trieu.le2@utoledo.edu}

\thanks{$^\dag$ Supported in part by MOE's AcRF Tier 1 grant M4011166.110 (RG24/13)}




\begin{abstract}
We study the structure of $\calN_p$-spaces in the ball. In particular, we show that any such space is Moebius-invariant and for $0<p\le n$, all $\calN_p$-spaces are different. Our results will be of important uses in the study of operator theory on $\calN_p$-spaces.
\end{abstract}

\maketitle


\section{Introduction}


\subsection{Basic notation and definitions}
Throughout the paper, $n$ is a positive integer. Let $\B$ be the open unit ball in $\C^n$ with $\SSS$ as its boundary. The space $\calO(\B)$ consists of all holomorphic functions in $\B$ with the compact-open topology. Banach and Hilbert spaces of holomorphic functions on $\B$ have attracted a great attention of researchers, including function theorists and operator theorists. The books \cite{HKZ, Zhu,ZZ} are excellent sources for information on these spaces, which include Hardy, Bloch, Bergman and Bergman-type spaces, among others.

A motivation of our work comes from the class of $Q_p$-spaces on the open unit disk $\D$ on the complex plane. The background on these spaces can be found in the book \cite{Xia}. In short, for $p>0$, the $Q_p$-space consists of functions in $\calO(\D)$ such that
$$
\sup_{a\in\D}\int_{\D}|f'(z)|^2(1-|\sigma_a(z)|^2)^p\;dA(z)<\infty.
$$
Here $\sigma_a(z)=(a-z)/(1-\conj{a}z)$ is the automorphism of $\D$ that changes $0$ and $a$; and $dA$ is the Lebesgue area measure on the plane, normalized so that $A(\D)=1$.
It is known that $Q_p$-spaces coincide with the classical Bloch space $\calB$ for $p\in(1,\infty)$; $Q_1$ is equal to BMOA, the space of holomorphic functions on $\D$ with bounded mean oscillation; and for $p\in(0,1)$, the $Q_p$-spaces are all different.

If, in the definition of the $Q_p$-space, $f'(z)$ is replaced by $f(z)$, then we have the so-called $\calN_p$-space in the unit disk $\D$, which was first introduced and studied in \cite{Pal} and then in \cite{Uek}. The space $\calN_p$ consists of functions in $\calO(\D)$ for which
$$
\sup_{a\in\D}\int_{\D}|f(z)|^2(1-|\sig_a(z)|^2)^p\;dA(z)<\infty.
$$

Some properties of $\calN_p$-spaces are: for $p>1$, the $\calN_p$-space coincides with the Begrman-type space $A^{-1}$ consisting of holomorphic functions on the disk for which $\sup_{z \in \D} |f(z)|(1-|z|^2)<\infty$; and for $p\in(0,1]$, the $\calN_p$-spaces are all different.

Several results on different properties of composition operators as well as weighted composition operators acting on $\calN_p$-spaces, or from $\calN_p$-spaces into Bergman-type spaces have been obtained in \cite{Pal, Uek}.

\subsection{$\calN_p$-spaces in the ball}

With the aim to generalize $\calN_p$-spaces to higher dimensions, namely in the unit ball $\B$ of $\C^n$, in \cite{HK1}, for $p>0$, the first two authors introduced the $\calN_p$-space of $\B$ as follows:
\begin{align*}
&\calN_p = \calN_p(\B)\\
& = \left\{f \in \calO(\B): \|f\|_p=\sup_{a \in \B} \left(\int_{\B} |f(z)|^2(1-|\Phi_a(z)|^2)^pdV(z)\right)^{1/2}<\infty\right\},
\end{align*}
where $dV(z)$ is the normalized volume measure over $\B$ and $\Phi_a \in \textrm{Aut}(\B)$ is the involutive automorphism that interchanges $0$ and $a \in \B$ (see, e.g. \cite[Chapter 2]{Rud}).

For $p>0$, we denote by $A^{-p}(\B)$ a Bergman-type space consisting of functions $f\in\calO(\B)$ for which $|f|_p=\sup_{z \in \B} |f(z)|(1-|z|^2)^p<\infty$. In \cite{HK1}, several basic properties of $\calN_p$-spaces have been proved, in connection with the Bergman-type spaces $A^{-q}$. In particular, an embedding theorem for $\calN_p$-spaces and $A^{-q}$ was obtained, together with other useful properties.

\begin{thm}\label{basic}\cite{HK1}
The following statements hold:
\begin{enumerate}
\item [(a)] For $p>q>0$, we have $H^{\infty}\hookrightarrow\calN_q\hookrightarrow\calN_p\hookrightarrow A^{-\frac{n+1}{2}}$.
\item [(b)] For $p>0$, if $p>2k-1, k\in(0,\frac{n+1}{2}]$, then $A^{-k}\hookrightarrow \calN_p$. In particular, when $p>n$, $\calN_p=A^{-\frac{n+1}{2}}$.
\item [(c)] $\calN_p$  is a functional Banach space with the norm $\|\cdot\|_p$, and moreover, its norm topology is stronger than the compact-open topology.
\item [(d)] For $0<p<\infty$, $\calB\hookrightarrow\calN_p$, where $\calB$ is the Bloch space in $\B$.
\end{enumerate}
\end{thm}

For a holomorphic self-mapping $\varphi$ of $\B$ and a holomorphic function $u: \B \to \C$, the linear operator $W_{u,\varphi}: \calO(\B) \to \calO(\B)$ is called a \emph{weighted composition operator} with symbols $u$ and $\varphi$ if
$$
W_{u, \varphi}(f)(z)=u(z) \cdot (f\circ\varphi(z)), f \in \calO(\B), z \in \B.
$$
Observe that if $u$ is identically $1$, then $W_{u, \varphi}=C_\varphi$ is the \emph{composition operator}, and if $\varphi$ is the identity, then $W_{u, \varphi}=M_u$ is the \emph{multiplication operator}. The books \cite{CM, Sha} are excellent sources on composition operators on analytic function spaces.

Considering weighted composition operators between $\calN_p$ and Bergman-type spaces $A^{-q}$, the properties above allowed us to prove criteria for boundedness and compactness of these operators \cite[Theorems 3.2 and 3.4]{HK1}. Furthermore, in \cite{HK2} the compact differences of weighted composition operators $W_{u,\varphi}$ acting from $\calN_p$-space to the space $A^{-q}$ were considered. Different properties stated in Theorem \ref{basic} were used.

The structure of this paper is as follows: in Section 2 we show that the space of multipliers of the $\calN_p$-space is precisely the space $H^\infty$ of bounded holomorphic functions on $\B$. We also show that $\calN_p$-space is a Moebius invariant, which is derived from the isometry property of some class of weighted composition operators. Section 3 deals with the closure of polynomials in $\calN_p$-spaces. Here we introduce the little space $\calN^0_p$ of $\calN_p$ which plays an important role in the proof of the density of polynomials in $\calN_p$. In Section 4 we establish criteria for a function to be in $\calN_p$ (respectively, in $\calN^0_p$) via $p$-Carleson measure (respectively, vanishing $p$-Carleson measure). Section 5 is devoted to the result that for small values of $p$ (that is $0<p\le n$), all $\calN_p$-spaces are different, and so the relationship between $\calN_p$-spaces is given completely. Here functions in the Hadamard gap class are used.

We remark that although the $N_p$-spaces are closely related to the $Q_p$-spaces, our approach and techniques used in the present paper are different from those for $Q_p$-spaces. Moreover, they have their own interests.

Throughout this paper, $d\sigma$ denotes the normalized surface measure on the boundary $\mathbb S$ of $\B$. For $a, b \in \R$, $a \lesssim b$ ($a \gtrsim b$, respectively) means there exists a positive number $C$, which is independent of $a$ and $b$, such that $a \leq Cb$ ($ a \geq Cb$, respectively). If both $a \lesssim b$ and $a \gtrsim b$ hold, we write $a \simeq b$.


\section{Multipliers and isometric weighted composition operators}


\subsection{Multipliers and $\calM$-invariance of $\calN_p$-spaces}

We first describe the space $\Mult(\Np)$ of multipliers of $\Np$. Recall that a function $u:\B\rightarrow\C$ is a \emph{multiplier} of $\Np$ if $uf$ belongs to $\Np$ for all $f$ in $\Np$. An application of the the closed graph theorem shows that for any $u\in\Mult(\Np)$, the multiplication operator $M_{u}$ is bounded on $\Np$.

\begin{prop}
\label{T:Multipliers}
For any $p>0$, we have $\Mult(\Np) = H^{\infty}$. Furthermore, for any $u\in H^{\infty}$, $\|M_{u}\| = \|u\|_{\infty}$.
\end{prop}

\begin{proof}
For $u\in H^{\infty}$ and $f\in\Np$, the function $uf$ belongs to $\calO(\B)$ and it follows immediately from the definition of the norm in $\Np$ that
\begin{align*}
\|uf\|_{p} \leq \|u\|_{\infty}\|f\|_{p}.
\end{align*}
This shows that $H^{\infty}\subset\Mult(\Np)$ and $\|M_u\|\leq\|u\|_{\infty}$. Now suppose that $u$ is an element in $\Mult(\Np)$. For any integer $m\geq 1$, we have
\begin{align*}
\|u^m\|_{p} = \|M_u^{m}1\|_{p} \leq \|M_u\|^{m}\|1\|_{p}.
\end{align*}
Combining with Theorem \ref{basic}, we obtain a positive constant $C>0$ independent of $u, m$ and $z$ such that for any $z\in\B$
\begin{align*}
|u^{m}(z)| \leq C(1-|z|^2)^{-(n+1)/2}\|u^m\|_{p} \leq C(1-|z|^2)^{-(n+1)/2}\|1\|_{p}\|M_u\|^{m}.
\end{align*}
Consequently,
\begin{align*}
|u(z)| \leq \Big(C(1-|z|^2)^{-(n+1)/2}\|1\|_{p}\Big)^{1/m}\|M_u\|.
\end{align*}
Letting $m\rightarrow\infty$, we conclude that $|u(z)|\leq \|M_{u}\|$ for all $z\in\B$. Therefore, $u$ belongs to $H^{\infty}$ and $\|u\|_{\infty}\leq\|M_u\|$. This completes the proof of the theorem.
\end{proof}

By using weighted composition operators with particular symbols, we obtain an alternate description of the norm in $\Np$.

For each $w \in \B$, set
\begin{equation}\label{probe}
k_w(z)=\left( \frac{1-|w|^2}{(1-\langle z,w \rangle)^2} \right)^{\frac{n+1}{2}},\ z \in \B.
\end{equation}
Such $k_w$ is a normalized reproducing kernel function in the Bergman space $A^2$. By \cite[Lemma 3.1]{HK1}, we have $k_w \in \calN_p$ and $\displaystyle\sup_{w\in\B}\|k_w\|_p\le1$. Note that for all $w\in\B$, we have $k_w(w)=(1-|w|^2)^{-(n+1)/2}$.

Furthermore, for any $\Phi\in\Aut(\B)$, by \cite[Theorem 2.2.5]{Rud}, there exists a unitary operator $U$ such that $\Phi = U\Phi_a$, where $a=\Phi^{-1}(0)$. This shows that $|\Phi(z)|=|\Phi_a(z)|$ for all $z\in\B$. Consequently, we obtain
\begin{align*}
\|f\|_{p} & = \sup_{a\in\B}\Big(\int_{\B}|f(z)|^2(1-|\Phi_a(z)|^2)^{p}\,dV(z)\Big)^{1/2}\\
& = \sup_{\Phi\in\Aut(\B)}\Big(\int_{\B}|f(z)|^2(1-|\Phi(z)|^2)^{p}\,dV(z)\Big)^{1/2}\\
& = \sup_{\Phi\in\Aut(\B)}\Big(\int_{\B}|f(z)|^2(1-|\Phi^{-1}(z)|^2)^{p}\,dV(z)\Big)^{1/2}.
\end{align*}

For $\Phi\in\Aut(\B)$, let $a=\Phi^{-1}(0)$. Denote by $W_{\Phi}$ the weighted composition operator $W_{k_a,\Phi}$. By the change of variables $z=\Phi(w)$ (see \cite[Theorem 2.2.6]{Rud}), we obtain
\begin{align*}
& \int_{\B}|f(z)|^2(1-|\Phi^{-1}(z)|^2)^{p}\,dV(z)\\
& \qquad\qquad = \int_{\B}|f(\Phi(w))|^2(1-|w|^2)^p\Big(\dfrac{1-|a|^2}{|1-\langle w,a\rangle|^2}\Big)^{n+1}\,dV(w)\\
& \qquad\qquad = \int_{\B}|f(\Phi(w))|^2|k_{a}(w)|^{2}(1-|w|^2)^p\,dV(w)\\
& \qquad\qquad = \|W_{\Phi}f\|^2_{A^2_p}.
\end{align*}
Here $A^2_p$ is the weighted Bergman space over $\B$ defined by
\begin{align*}
A^2_p:=\left\{f\in\mathcal{O}(\B): \|f\|_{A^2_p} = \Big(\int_{\B}|f(z)|^2(1-|z|^2)^p\,dV(z)\Big)^{1/2}<\infty \right\}.
\end{align*}
Combing the above formulas, we have
\begin{align}
\label{Eqn:normNp}
\|f\|_{p} & = \sup\Big\{\|W_{\Phi}f\|_{A^2_p}: \Phi\in\Aut(\B)\Big\}.
\end{align}

It can be checked by a direct calculation that for any two automorphisms $\Phi$ and $\Psi$ in $\Aut(\B)$, there exists a  complex number $\lambda$ with modulus one such that $W_{\Phi}W_{\Psi} = \lambda W_{\Psi\circ\Phi}$. Consequently,
\begin{align*}
\|W_{\Psi}f\|_{p} & = \sup\Big\{\|W_{\Phi}W_{\Psi}f\|_{A^2_p}: \Phi\in\Aut(\B)\Big\}\\
& = \sup\Big\{\|W_{\Psi\circ\Phi}f\|_{A^2_p}: \Phi\in\Aut(\B)\Big\} = \|f\|_{p}.
\end{align*}
The argument above proves the following result.

\begin{thm}
\label{T:isometricWCOs}
For any automorphism $\Psi$ of the unit ball $\B$, the weighted composition operator $W_{\Psi}$ is a surjective isometry on $\mathcal{N}_p$.
\end{thm}

Recall that a space $\mathcal{X}$ of functions defined on $\B$ is said to be \emph{Moebius-invariant}, or simply \emph{$\calM$-invariant}, if $f \circ \Phi \in \mathcal{X}$ for every $f \in \mathcal{X}$ and every $\Phi \in \textrm{Aut}(\B)$ (see, e.g., \cite{Rud}). As a corollary to Theorem \ref{T:isometricWCOs}, we obtain

\begin{cor}
The space $\calN_p$ is $\calM$-invariant. Moreover, for any $\Phi \in \Aut(\B)$, we have
\begin{equation} \label{normPsi}
\|C_\Phi\|=\|M_{1/k_a}\| = \left( \frac{1+|a|}{1-|a|}\right)^{\frac{n+1}{2}},
\end{equation}
where $a=\Phi^{-1}(0)$.
\end{cor}

\begin{proof}
Note that for any automorphism $\Phi$ on $\B$, we have $C_\Phi=M_{1/k_a} \circ W_{\Phi}$. Since $1/k_a$ is a bounded function, it is a multiplier of $\calN_p$ and since $W_{\Phi}$ is a surjective isometry, it follows that $\|C_\Phi\|=\|M_{1/k_a}\|$. The last equality in \eqref{normPsi} follows from the fact that $\|1/k_a\|_\infty = \Big(\frac{1+|a|}{1-|a|}\Big)^{\frac{n+1}{2}}$ and Theorem \ref{T:Multipliers}.
\end{proof}

\subsection{Upper estimate of $\|\cdot\|_p$ for small $p$}

It is immediate from the definition of the norm in $\Np$ that $\|f\|_{p}\geq \|f\|_{A^2_p}$ for all $p>0$. As the last result in this section, we provide an upper estimate for $\|f\|_{p}$ when $p\leq n$.
\begin{prop}
\label{L:normEstimateNp}
For $0<p\leq n$, there exists a positive constant $C=C(n,p)$ such that for any $f\in\Np$ we have
\begin{align*}
\|f\|_{p} \leq C\left(\int_{\B}\big(\sup_{|w|=|z|}|f(w)|^2\big)(1-|z|^2)^p\,dV(z)\right)^{1/2}.
\end{align*}
\end{prop}

\begin{proof}
For $a\in\B$, integration in polar coordinates (\cite[Lemma 1.8]{Zhu}) gives
\begin{align*}
&\int_{\B}|f(z)|^2(1-|\Phi_a(z)|^2)^{p}\,dV(z)\\
&\quad = \int_{\B}|f(z)|^2\dfrac{(1-|z|^2)^p(1-|a|^2)^p}{|1-\langle z,a\rangle|^{2p}}\,dV(z)\\
&\quad = 2n\int_{0}^{1}r^{2n-1}(1-|r|^2)^{p}\Big(\int_{\mathbb{S}}|f(r\zeta)|^2\dfrac{(1-|a|^2)^p}{|1-\langle\zeta,ra\rangle|^{2p}}d\sigma(\zeta)\Big)dr\\
&\quad \leq 2n\int_{0}^{1}r^{2n-1}(1-|r|^2)^{p}\big(\sup_{\zeta\in\mathbb{S}}|f(r\zeta)|^2\big)\Big(\int_{\mathbb{S}}\dfrac{(1-|a|^2)^p}{|1-\langle\zeta,ra\rangle|^{2p}}d\sigma(\zeta)\Big)dr.
\end{align*}
Now \cite[Theorem 1.12]{Zhu} with $a\in\B$ and $0<r<1$ gives
\begin{align*}
\int_{\mathbb S} \frac{d\sigma(\zeta)}{|1-\langle r\zeta, a\rangle|^{2p}}
& = \int_{\mathbb S} \frac{d\sigma(\zeta)}{|1-\langle \zeta, ra\rangle|^{2p}}
 = \int_{\mathbb S} \frac{d\sigma(\zeta)}{|1-\langle \zeta, ra\rangle|^{n+(2p-n)}}\\
&\simeq \begin{cases}
\textrm{bounded in} \ \B \ & \text{ for } 0<p<\frac{n}{2},\\
\log\frac{1}{1-r^2|a|^2} \le \log\frac{1}{1-|a|^2} \ & \text{ for } p=\frac{n}{2},\\
(1-r^2|a|^2)^{n-2p} \le (1-|a|^2)^{n-2p} \ & \text{ for } \frac{n}{2}<p\leq n.
\end{cases}
\end{align*}
Thus, for all cases of $0<p\leq n$, there exists a positive constant $C$ independent of $a$ and $r$ such that
$$
\int_{\mathbb S} \frac{(1-|a|^2)^{p}}{|1-\langle r\zeta, a\rangle|^{2p}}d\sigma(\zeta) \le C.
$$
It then follows that
\begin{align*}
&\int_{\B}|f(z)|^2(1-|\Phi_a(z)|^2)^{p}\,dV(z)\\
&\quad\quad\leq C\Big(2n\int_{0}^{1}r^{2n-1}(1-|r|^2)^{p}\sup_{|w|=r}|f(w)|^2\,dr\Big)\\
&\quad\quad = C\int_{\B}\big(\sup_{|w|=|z|}|f(w)|^2\big)(1-|z|^2)^p\,dV(z).
\end{align*}
Taking supremum over $a\in\B$ gives the required inequality.
\end{proof}


\section{The closure of all polynomials in $\calN_p$}


It is natural to consider what the closure of all the polynomials is in $\calN_p$-spaces. We introduce the little space $\calN_p^0$ of $\calN_p$, which is defined as
$$
\calN_p^0=\calN_p^0(\B)=\left\{ f \in \calN_p: \lim_{|a| \to 1^{-}} \int_\B |f(z)|^2(1-|\Phi_a(z)|^2)^pdV(z)=0\right\}.
$$
In this section, we show that the closure of all the polynomials on $\B$ coincides with the little space $\calN_p^0$.

\begin{prop}
$\calN_p^0$ is a closed subspace of $\calN_p$, and hence it is a Banach space.
\end{prop}

\begin{proof}
It can be easily shown that $\calN^0_p$ is a subspace of $\calN_p$ and hence it suffices to show that $\calN_p^0$ is closed.

Consider a sequence $\{f_n\} \subset \calN_p^0$ that converges to some $f \in \calN_p$. We show that $f \in \calN_p^0$. Indeed, for any $\varepsilon>0$, there exists an $N \in \N$, such that  $\|f-f_n\|_p<\sqrt{\frac{\varepsilon}{4}},\ \forall n>N$. Let $n_0>N$ be fixed. Since $f_{n_0} \in \calN^0_p$, there exists a $\del \in (0,1)$, such that
$$
\sup_{\del<|a|<1} \int_{\B} |f_{n_0}(z)|^2(1-|\Phi_a(z)|^2)^pdV(z)<\frac{\varepsilon}{4}.
$$
As a consequence,
\begin{eqnarray*}
&&\sup_{\del<|a|<1} \int_{\B} |f(z)|^2(1-|\Phi_a(z)|^2)^pdV(z)\\
&&\le\sup_{\del<|a|<1} \int_{\B} 2\left(|f(z)-f_{n_0}(z)|^2+|f_{n_0}(z)|^2\right)(1-|\Phi_a(z)|^2)^pdV(z)\\
&&\le2\|f-f_{n_0}\|^2_p+2\sup_{\del<|a|<1} \int_{\B} |f_{n_0}|^2(1-|\Phi_a(z)|^2)^pdV(z)<\varepsilon,
\end{eqnarray*}
which implies
$$
\lim_{|a| \to 1^{-}} \int_{\B} |f(z)|^2(1-|\Phi_a(z)|^2)^pdV(z)=0.
$$
From this, we conclude that $\calN_p^0$ is closed.
\end{proof}

\begin{lem}
\label{L:subspaceLittleNp}
For any $p>0$, we have $A^2\subset\Np^{0}$.
\end{lem}

\begin{proof}
Let $f$ be an element in $A^2$. For $a\in\B$, define $$g_a(z) = |f(z)|^2(1-|\Phi_a(z)|^2)^{p} = |f(z)|^2\dfrac{(1-|a|^2)^p(1-|z|^2)^p}{|1-\langle z,a\rangle|^{2p}},\quad z\in\B.$$
We have $0\leq g_a(z)\leq |f(z)|^2$ and $\lim_{|a|\to 1^{-}}|g_a(z)|=0$ for all $z\in\B$. The Dominated Convergence Theorem then implies
\begin{align*}
\lim_{|a|\to 1^{-}}\int_{\B}|f(z)|^2(1-|\Phi_a(z)|^2)^p\;dV(z) = \lim_{|a|\to 1^{-}}\int_{\B}g_a(z)\;dV(z) = 0.
\end{align*}
This shows that $f$ belongs to $\Np^{0}$.
\end{proof}

\begin{thm} \label{criLNp}
Suppose $f \in \calN_p$. Then $f \in \calN^0_p$ if and only if
$$
\|f_r-f\|_p \to 0 \, \text{ as }\, r \to 1^{-},
$$
where $f_r(z)=f(rz)$ for all $z \in \B$.
\end{thm}

\begin{proof}
$\bullet$ \textbf{Necessity}. Suppose $f \in \calN_p^0$. This implies that for any $\eps>0$, there exists $\del>0$ such that with $\del<|a|<1$, we have
\begin{equation} \label{little01}
\int_{\B} |f(z)|^2(1-|\Phi_a(z)|^2)^p\;dV(z)<\frac{\eps}{6\cdot 4^n}.
\end{equation}

Furthermore, by Schwarz-Pick Lemma (see, e.g., \cite[Theorem 8.1.4]{Rud}), we have
\begin{equation} \label{little02}
|\Phi_{ra}(rz)| \le |\Phi_a(z)|\ \hbox{for all}\ r\in(0,1)\ \hbox{and}\ a, z\in\B.
\end{equation}

Now take and fix $\del_0\in(\del,1)$. Consider $r$ satisfying $\max\left\{\frac{1}{2},\frac{\del}{\del_0}\right\}<r<1$. In this case, for all $a\in\B$ with $|a|\in(\del_0,1)$, by \eqref{little01} and \eqref{little02}, we have
\begin{align*}
&\int_{\B} |f(rz)|^2(1-|\Phi_a(z)|^2)^p\;dV(z)\\
&\qquad\qquad\qquad\leq \int_{\B} |f(rz)|^2(1-|\Phi_{ra}(rz)|^2)^p\;dV(z)\\
&\qquad\qquad\qquad = \left(\frac{1}{r}\right)^{2n}\int_{r\B} |f(w)|^2(1-|\Phi_{ra}(w)|^2)^p\;dV(w)\\
&\qquad\qquad\qquad \leq 4^n \int_{\B} |f(w)|^2(1-|\Phi_{ra}(w)|^2)^p\;dV(w)<\frac{\eps}{6}.
\end{align*}

On the other hand, since $f\in A^2_{p}$, by \cite[Proposition 2.6]{Zhu}, $f(rz)$ converges to $f(z)$ as $r\to1^-$, in the norm topology of the Bergman space $A^2_p(\B)$. This implies that there exists a $r_1\in(0,1)$ such that for $r_1<r<1$, we have
$$
\int_{\B} |f(rz)-f(z)|^2(1-|z|^2)^p\;dV(z)<\frac{(1-\del_0)^{2p} \cdot \eps}{3}.
$$
Consequently, for $|a|\le\del_0$ and $r_1<r<1$, we have
\begin{eqnarray*}
&&\sup_{|a|\le\del_0}\int_{\B} |f(rz)-f(z)|^2(1-|\Phi_a(z)|^2)^p\;dV(z)\\
&&= \sup_{|a|\le\del_0}\left\{ (1-|a|^2)^p \int_{\B} |f(rz)-f(z)|^2 \frac{(1-|z|^2)^p}{|1-\langle z,a \rangle|^{2p}}\;dV(z)\right\}\\
&&\leq \sup_{|a|\le\del_0}\int_{\B} |f(rz)-f(z)|^2 \frac{(1-|z|^2)^p}{|1-\langle z, a \rangle|^{2p}}\;dV(z) \\
&& \leq \frac{1}{(1-\del_0)^{2p}} \int_{\B} |f(rz)-f(z)|^2(1-|z|^2)^p\;dV(z)<\frac{\eps}{3}.
\end{eqnarray*}
For all $r$ with $\max\left\{\frac{1}{2},\frac{\del}{\del_0},r_1\right\}<r<1$, combing the above estimates yields
\begin{eqnarray*}
\|f_r-f\|_p^2%
&=&\sup_{a \in \B} \int_{\B} |f(rz)-f(z)|^2(1-|\Phi_a(z)|^2)^p\;dV(z)\\
&\leq& \left(\sup_{|a| \leq \del_0}+\sup_{\del_0<|a|<1}\right) \int_{\B} |f(rz)-f(z)|^2(1-|\Phi_a(z)|^2)^p\;dV(z)\\
&\leq& \frac{\eps}{3}+2\sup_{\del_0<|a|<1} \int_{\B}\left(|f(rz)|^2+|f(z)|^2\right)(1-|\Phi_a(z)|^2)^p\;dV(z)\\
&<& \frac{\eps}{3}+2\left(\frac{\eps}{6}+\frac{\eps}{6\cdot4^n}\right)<\frac{\eps}{3}+2\left(\frac{\eps}{6}+\frac{\eps}{6}\right)=\eps,
\end{eqnarray*}
which show that $\|f_r-f\|_p \to 0$ as $r \to 1^{-}$.

\medskip
$\bullet$ \textbf{Sufficiency}. Suppose that $\|f_r-f\|_p\to 0$, as $r\to 1^{-}$. For each $0<r<1$, the holomorphic function $f_r$ is bounded, hence it belongs to $\Np^{0}$ by Lemma \ref{L:subspaceLittleNp}. Since $\calN_p^0$ is a closed subspace of $\calN_p$, it follows that $f$ belongs to $\calN_p^0$.
\end{proof}

As a corollary to Theorem \ref{criLNp}, we obtain

\begin{cor}
The set of polynomials is dense in $\calN_p^0$.
\end{cor}

\begin{proof}
By Theorem \ref{criLNp}, for any $f \in \calN_p^0$, we have
$$
\lim_{r\to1^{-}}\|f_r-f\|_p = 0.
$$
Since each $f_r$ can be uniformly approximated by polynomials, and moreover, by Theorem \ref{basic}, the sup-norm dominates the $\calN_p$-norm, we conclude that every $f \in \calN_p^0$ can be approximated in the $\calN_p$-norm by polynomials.
\end{proof}



\section{$\calN_p$-norm via Carleson measures}


Recall (see, e.g., \cite{Zhu}) that for $\xi\in\SSS$ and $r>0$, a Carleson tube at $\xi$ is defined as
$$
Q_r(\xi)=\{z\in\B: |1-\langle z,w\rangle|<r\}.
$$

A positive Borel measure $\mu$ in $\B$ is called a \emph{$p$-Carleson measure} if there exists a constant $C>0$ such that
$$
\mu(Q_r(\xi)) \le Cr^{p}
$$
for all $\xi\in\SSS$ and $r>0$. Moreover, if
$$
\lim_{r\to0} \frac{\mu(Q_r(\xi))}{r^p}=0
$$
uniformly for $\xi \in \SSS$, then $\mu$ is called a \emph{vanishing $p$-Carleson measure}.

\medskip
The following result describes a relationship between functions in $\calN_p$ as well as $\calN_p^0$ and Carleson measures.
\begin{prop}
Let $p>0$ and $f\in\calO(\B)$. Define $d\mu_{f,p}(z)=|f(z)|^2(1-|z|^2)^p\;dV(z)$. The following assertions hold.
\begin{enumerate}
\item $f\in\calN_p$ if and only if $\mu_{f,p}$ is a $p$-Carleson measure.
\item $f\in\calN_p^0$ if and only if $\mu_{f,p}$ is a vanishing $p$-Carleson measure.
\end{enumerate}
Moreover, it holds
\begin{equation} \label{eqnorm1}
\|f\|_p^2 \simeq \sup_{r \in (0, 1), \xi\in\SSS} \frac{\mu_{f, p}(Q_r(\xi))}{r^p} = \sup_{r\in(0,1),\xi\in\SSS} \frac{1}{r^p}\int_{Q_r(\xi)}|f(z)|^2(1-|z|^2)^p\;dV(z).
\end{equation}
\end{prop}

\begin{proof}
\noindent For any for $f\in\calO(\B)$, we can write
\begin{eqnarray*}
\|f\|_p^2%
&=& \sup_{a \in \B} \int_\B |f(z)|^2 \frac{(1-|a|^2)^p(1-|z|^2)^p}{|1-\langle a, z \rangle|^{2p}}\;dV(z)\\
&=& \sup_{a \in \B} \int_\B \left(\frac{1-|a|^2}{|1-\langle a, z \rangle|^2}\right)^pd\mu_{f, p}(z).
\end{eqnarray*}
Then statement (1) as well as \eqref{eqnorm1} follow from \cite[Theorem 45]{ZZ}.

\noindent On the other hand, statement (2) is a consequence of the ``little-oh version'' of \cite[Theorem 45]{ZZ}, which we provide a detailed proof below.
\end{proof}

\begin{lem}\label{l-oh}
Let $p=n+1+\alpha>0$ and $\mu$ be a finite positive Borel measure on $\B$. Then the following conditions are equivalent.
\begin{enumerate}
\item[(a)] $\mu$ is a vanishing $p$-Carleson measure.
\item[(b)] For each $s>0$,
\begin{equation} \label{eq2}
\lim_{|z| \to 1^{-}} \int_{\B} \frac{(1-|z|^2)^sd\mu(w)}{|1-\langle z, w \rangle|^{p+s}}=0.
\end{equation}
\item[(c)] For some $s>0$, \eqref{eq2} holds.
\end{enumerate}
\end{lem}

\begin{proof}
$\bullet$ The implication $(b)\implies (c)$ is obvious.

$\bullet\ (c) \implies (a)$: Suppose the condition (c) holds. This means that there exists $s>0$, such that
$$
\lim_{|z| \to 1^{-}} \int_{\B} \frac{(1-|z|^2)^sd\mu(w)}{|1-\langle z, w \rangle|^{p+s}}=0.
$$
Then for any $\varepsilon>0$, there exsits $\del>0$, such that when $\del<|z|<1$,
\begin{equation} \label{eq7}
\int_{\B} \frac{(1-|z|^2)^sd\mu(w)}{|1-\langle z, w \rangle|^{p+s}}<\varepsilon.
\end{equation}

We first show that $\mu$ is a $p$-Carleson measure. Indeed, for $|z| \leq \del$, we have
\begin{align*}
\int_\B  \frac{(1-|z|^2)^sd\mu(w)}{|1-\langle z, w \rangle|^{p+s}}
& \leq \int_\B  \frac{d\mu(w)}{|1-\langle z, w \rangle|^{p+s}}
\leq \frac{\mu(\B)}{(1-\del)^{p+s}}<\infty.
\end{align*}
This fact and \eqref{eq7} show that $\mu$ satisfies condition (c) in \cite[Theorem 45]{ZZ}. Consequently, $\mu$ is a $p$-Carleson measure.

Next we prove that $\mu$ is a vanishing $p$-Carleson measure. Let $\xi$ be in $\SSS$. For $r\in (0,1-\delta)$, put $z=(1-r)\xi$. Then $\delta<|z|<1$ and for any $w\in\Q_r(\xi)$,
\begin{align*}
|1-\langle z,w\rangle| & = \Big|(1-r)\big(1-\langle \xi,w\rangle\big)+r\Big| \leq (1-r)r+r < 2r.
\end{align*}
Consequently,
\begin{align*}
\dfrac{(1-|z|^2)^s}{|1-\langle z,w\rangle|^{p+s}}\geq \dfrac{(1-|z|)^s}{|1-\langle z,w\rangle|^{p+s}} \geq \dfrac{r^s}{(2r)^{p+s}} = \dfrac{2^{-(p+s)}}{r^{p}}.
\end{align*}
Using \eqref{eq7}, we obtain
\begin{align*}
\frac{\mu(Q_r(\xi))}{r^{p}} & = \frac{1}{r^{p}} \int_{Q_r(\xi)} d\mu(w)
\leq 2^{p+s} \int_{Q_r(\xi)} \frac{(1-|z|^2)^sd\mu(w)}{|1-\langle z, w \rangle|^{p+s}}\\
 &\leq 2^{p+s} \int_{\B} \frac{(1-|z|^2)^sd\mu(w)}{|1-\langle z, w \rangle|^{p+s}}<2^{p+s} \varepsilon,
\end{align*}
which implies (a).

$\bullet \ (a)\implies (b)$: Suppose (a) holds, which means that for any $\varepsilon>0$, there exists $\del>0$ such that for $0<r<\del$, we have
\begin{equation} \label{eq5}
\frac{\mu(Q_r(\xi))}{r^{p}}<\varepsilon \quad\text{ for all}\ \xi \in \SSS.
\end{equation}
Also, since $\mu$ is a Carleson measure, there is a positive constant $C$ such that
\begin{align}
\label{Eqn:CarlesonMeasure}
\mu(Q_{r}(\xi)) \leq C r^{p}\quad\text{ for all } \xi\in\SSS \text{ and } 0<r<1.
\end{align}

Let $s>0$. For the same $\varepsilon$ chosen above, take $N_0 \in \N$ such that
\begin{equation} \label{eq4}
\sum_{k=N_0+1}^{\infty} \frac{1}{2^{sk}}<\varepsilon.
\end{equation}

Take and fix some $z \in \B$ with $\max\left\{\frac{3}{4}, 1-\frac{\del}{2^{N_0+1}}\right\}<|z|<1$ and set $\xi=z/|z|$. For any nonnegative integer $k$, let $r_k=2^{k+1}(1-|z|)$. We decompose the unit ball $\B$ into the disjoint union of the following sets:
$$
E_0=Q_{r_0}(\xi), \quad E_k=Q_{r_k}(\xi)\setminus Q_{r_{k-1}}(\xi), \quad  1 \leq k< \infty.
$$

For $k \geq 2$ and $w \in E_k$, we have
\begin{align}
\label{Eqn:w_in_Ek}
|1-\langle z, w\rangle|& = \big||z|(1-\langle \xi, w\rangle)+(1-|z|)\big|\notag\\
&\geq |z||1-\langle \xi, w\rangle|-(1-|z|)\notag\\
&\geq (3/4)2^k(1-|z|)-(1-|z|)\notag\\
&\geq 2^{k-1}(1-|z|).
\end{align}
This also holds for $k=1$ and $k=0$, because
$$
|1-\langle z, w \rangle| \geq 1-|z| \geq \frac{1}{2}(1-|z|).
$$
Now we consider two cases of $k \in \N$.

- Case I: $0 \leq k \leq N_0$. In this case, we have
$$
0<r_k=2^{k+1}(1-|z|) \leq 2^{N_0+1}(1-|z|)<\del.
$$
This implies, by \eqref{eq5}, that
\begin{equation} \label{eq6}
\mu(E_k) \leq  \mu(Q_{r_k}(\xi)) \leq r_k^p\,\varepsilon = 2^{p(k+1)}(1-|z|)^{p}\varepsilon.
\end{equation}

- Case II: $k>N_0$. Using \eqref{Eqn:CarlesonMeasure}, we have
\begin{equation} \label{eq9}
\mu(E_k) \leq \mu(Q_{r_k}(\xi)) \leq 2^{p(k+1)}(1-|z|)^{p}C.
\end{equation}

For $\max\left\{\frac{3}{4}, 1-\frac{\del}{2^{N_0+1}}\right\}<|z|<1$, using \eqref{eq4}, \eqref{Eqn:w_in_Ek}, \eqref{eq6} and \eqref{eq9}, we compute
\begin{align*}
& \int_{\B} \frac{(1-|z|^2)^sd\mu(w)}{|1-\langle z, w \rangle|^{p+s}} = \sum_{k=0}^{\infty} \int_{E_k} \frac{(1-|z|^2)^sd\mu(w)}{|1-\langle z, w \rangle|^{p+s}}\\
& \leq \sum_{k=0}^{\infty} \frac{(1-|z|^2)^s\mu(E_k)}{(2^{k-1}(1-|z|))^{p+s}} = \left(\sum_{k=0}^{N_0}+\sum_{k=N_0+1}^{\infty}\right)\frac{(1-|z|^2)^s\mu(E_k)}{(2^{k-1}(1-|z|))^{p+s}}\\
& \leq \sum_{k=0}^{N_0} \frac{2^{s+p(k+1)}(1-|z|)^{p+s}\varepsilon}{2^{(k-1)(p+s)}(1-|z|)^{p+s}}+ \sum_{k=N_0}^\infty \frac{2^{s+p(k+1)}(1-|z|)^{p+s}C}{2^{(k-1)(p+s)}(1-|z|)^{p+s}}\\
& = \varepsilon\sum_{k=0}^{N_0} \frac{2^{s+p(k+1)}}{2^{(k-1)(p+s)}}+ C\sum_{k=N_0+1}^\infty \frac{2^{s+p(k+1)}}{2^{(k-1)(p+s)}}\\
& = \varepsilon\cdot 4^{s+p}\sum_{k=0}^{N_0} \frac{1}{2^{ks}} + C\cdot 4^{p+s}\sum_{k=N_0+1}^{\infty}\frac{1}{2^{ks}}\\
& \leq \frac{4^{s+p}}{1-2^{-s}}\varepsilon + 4^{p+s}C\varepsilon = M\varepsilon,
\end{align*}
where $M$ is a constant depending only on $s$ and $p$. This shows that statement (b) holds.
\end{proof}


\section{Differences of $\calN_p$ for small values of $p$}


In this section, we show that for $p$ small, that is $0<p\le n$, all $\calN_p$-spaces are different. This together with Theorem \ref{basic} (b) gives a complete relationship between $\calN_p$-spaces for all $p>0$.

We prove this fact by a construction. In \cite{RW}, the authors constructed a sequence of homogeneous polynomials $(P_k)_{k \in \N}$ satisfying $\deg(P_k)=k$,
\begin{equation}
\label{RWprobe}
\|P_k\|_\infty=\sup_{\xi \in \SSS} |P_k(\xi)|=1, \ \textrm{and} \left( \int_\SSS |P_k(\xi)|^2d\sigma(\xi) \right)^{1/2} \ge \frac{\sqrt{\pi}}{2^n}.
\end{equation}
Note that the homogeneity of $P_k$ implies that $|P_k(z)|\leq |z|^k$ for all $z\in\B$.

Let $\{m_{k}\}_{k=0}^{\infty}$ be a sequence of positive integers such that $m_{k+1}/m_{k}\geq c$ for all $k\geq 0$, where $c>1$ is a constant. Let
\begin{align}
\label{Eqn:HadamardGap}
f(z) = \sum_{k=0}^{\infty}b_kP_{m_k}(z)\ \text{ for } z\in\B.
\end{align}
Such a function is said to belong to the \emph{Hadamard gap class}. A characterization for a Hadamard gap class function to be in a weighted Bergman space was given in \cite{SS}. In the following result, we obtain an estimate for the $\calN_p$-norm and $A^{-q}$-norm of $f$. These results are higher dimensional versions of \cite[Theorem 3.3]{Pal}.

\begin{thm}
\label{T:HadamardGapClass}
Let $f$ be defined as in \eqref{Eqn:HadamardGap}. Let $p$ be a positive real number. Then the following statements hold:
\begin{itemize}
\item[(a)] For $0<p\leq n$, we have $\|f\|_{p}^2\simeq\sum_{k=0}^{\infty}\frac{|b_k|^2}{m_k^{p+1}}$.
\item[(b)] For any $q>0$, we have $|f|_q\simeq\sup_{k}\frac{|b_k|}{m_k^{q}}$.
\end{itemize}
(Here, $\|f\|_p$ and $|f|_q$ denote the norm of $f$ in the spaces $\calN_p$ and $A^{-q}$, respectively).
\end{thm}

Note that Theorem \ref{T:HadamardGapClass}, for $n=1$, contains the corresponding results in \cite{Pal} as particular cases.

\begin{proof}
(a) Consider $0<p\leq n$. Since $|P_{m_k}(w)|\leq |w|^{m_k}$ for all $k\geq 0$ and $w\in\B$, we have
$$
\sup_{|w|=|z|}|f(w)| \leq\sum_{k=0}^{\infty}|b_k||z|^{m_k}
$$
for any $z\in\B$. Proposition \ref{L:normEstimateNp} and integration in polar coordinates then give
\begin{align*}
\|f\|_{p}^2 \lesssim \int_{0}^{1}\Big(\sum_{k=0}^{\infty} |b_k|r^{m_k}\Big)^2(1-r^2)^p\,dr.
\end{align*}
On the other hand, by \cite[Theorem 1]{MM},
\begin{align*}
\int_{0}^{1}\Big(\sum_{k=0}^{\infty} |b_k|r^{m_k}\Big)^2(1-r^2)^p\,dr \simeq \sum_{k=0}^{\infty}2^{-k(p+1)}\Big(\sum_{2^k\leq m_j <2^{k+1}}|b_j|\Big)^{2}.
\end{align*}
Since $m_{j+1}\geq c\,m_j$ for all $j$, the cardinality of $\{j: 2^{k}\leq m_j <2^{k+1}\}$ is at most $1+\log_{c}2$. It then follows that
\begin{align*}
\sum_{k=0}^{\infty}2^{-k(p+1)}\Big(\sum_{2^k\leq m_j <2^{k+1}}|b_j|\Big)^{2} &
\lesssim \sum_{k=0}^{\infty}2^{-k(p+1)}\Big(\sum_{2^k\leq m_j < 2^{k+1}}|b_j|^2\Big)\\
&\lesssim \sum_{k=0}^{\infty}\Big(\sum_{2^k\leq m_j <2^{k+1}}m_j^{-(p+1)}|b_j|^2\Big)\\
& =\sum_{k=0}^{\infty}\frac{|b_k|^2}{m_k^{p+1}}.
\end{align*}
Combining the above estimates, we obtain
$
\displaystyle\|f\|_p^2  \lesssim \sum_{j=0}^{\infty}\dfrac{|b_k|^2}{m_k^{p+1}}.
$

To prove the reverse inequality, we use the orthogonality of homogeneous polynomials of different degrees in $A^2_p$ to obtain
\begin{align*}
\|f\|_{p}^2 \geq \|f\|^2_{A^2_p} & =\int_{\B}\Big|\sum_{k=0}^{\infty}b_kP_{m_k}(z)\Big|^2 (1-|z|^2)^p\,dV(z)\\
& = \sum_{k=0}^{\infty}|b_k|^2\int_{\B}|P_{m_k}(z)|^2(1-|z|^2)^p\,dv(z)\\
& = \sum_{k=0}^{\infty}|b_k|^2\int_{0}^{1}2nr^{2n+2m_k-1}(1-r^2)^p\,dr\int_{\SSS}|P_{m_k}(\xi)|^2\,d\sigma(\xi)\\
& \gtrsim\sum_{k=0}^{\infty}|b_k|^2\int_{0}^{1} t^{n+m_k-1}(1-t)^p\,dt\\
& \text{(by \eqref{RWprobe} and the change of variables $t=r^2$)}\\
& =\sum_{k=0}^{\infty}|b_k|^2\dfrac{\Gamma(n+m_k)\Gamma(p+1)}{\Gamma(n+m_k+p+1)} \gtrsim\sum_{k=0}^{\infty}\dfrac{|b_k|^2}{m_k^{p+1}}.
\end{align*}
The last inequality follows from Stirling's formula. We have thus completed the proof of (a).

(b) Assume $f\in A^{-q}$ for $q>0$. Fix a positive integer $k$. For $r>0$ and $\xi\in\SSS$, we have
$|f|^2_q(1-r)^{-2q}\geq |f(r\xi)|^2$. Integrating with respect to $\xi\in\SSS$ and using \eqref{RWprobe} yield
\begin{align*}
\dfrac{|f|^2_q}{(1-r)^{2q}} & \geq \int_{\SSS}|f(r\xi)|^2\,d\sigma(\xi) =\int_{\SSS}\Big|\sum_{j=0}^{\infty}b_jr^{m_j}P_{m_j}(\xi)\Big|^2\,d\sigma(\xi)\\
& = \sum_{j=0}^{\infty}|b_j|^2r^{2m_j}\int_{\SSS}|P_{m_j}(\xi)|^2\,d\sigma(\xi)\gtrsim |b_k|^2r^{2m_k}.
\end{align*}
Setting $r=m_k/(q+m_k)$, we obtain
\begin{align*}
|b_k| \lesssim \dfrac{|f|_q}{r^{m_k}(1-r)^q} = |f|_q\Big(1+\frac{q}{m_k}\Big)^{m_k}\Big(1+\frac{m_k}{q}\Big)^{q} \lesssim |f|_qm_k^{q},
\end{align*}
which implies $\sup_{k}\dfrac{|b_k|}{m_k^q} \lesssim |f|_q$.

Put $L=\sup_{k}\big\{|b_k|m_k^{-q}\big\}$ so $|b_k|\leq L\,m_k^{q}$ for all $k\geq 0$. For each $z \in \B$, we have
\begin{align}
\label{Eqn:fz}
\frac{|f(z)|}{1-|z|}%
&\le \Big( \sum_{k=0}^\infty |b_k||P_{m_k}(z)| \Big) \Big( \sum_{s=0}^\infty |z|^s\Big)\notag\\
&\le L\Big( \sum_{k=0}^\infty m_k^{q}|z|^{m^k} \Big) \Big( \sum_{s=0}^\infty |z|^s\Big)\\
&\quad\text{(since $|P_{m_k}(z)|\leq |z|^{m_k}$ for all $k\geq 0$)}\notag\\
& = L\sum_{\ell=1}^\infty \Big( \sum_{m_k \le\ell} m_k^{q}\Big) |z|^{\ell}.\notag
\end{align}
Since $m_{k+1}/m_k\geq c>1$ for all $k$, we have
\begin{align*}
\sum_{m_k\leq\ell}m_k^q & =\ell^{q}\sum_{m_k\leq\ell}\Big(\dfrac{m_k}{\ell}\Big)^q
 \leq \ell^q\sum_{s=0}^{\infty}(c^{-q})^{s} = \dfrac{\ell^q}{1-c^{-q}}.
\end{align*}
By Stirling's formula, it follows that
\begin{align}
\label{Eqn:SumMk}
\sum_{m_k\leq\ell}m_k^q & \leq \dfrac{\ell^q}{1-c^{-q}} \lesssim \dfrac{1}{1-c^{-q}}\dfrac{\Gamma(\ell+q+1)}{\Gamma(\ell+1)\,\Gamma(q)}.
\end{align}
Combing \eqref{Eqn:fz} and \eqref{Eqn:SumMk} yields
\begin{align*}
\frac{|f(z)|}{1-|z|} & \lesssim\dfrac{L}{1-c^{-q}}\sum_{\ell=0}^{\infty}\dfrac{\Gamma(\ell+q+1)}{\Gamma(\ell+1)\,\Gamma(q)}|z|^{\ell} = \dfrac{L}{1-c^{-q}}\dfrac{1}{(1-|z|)^{q+1}}.
\end{align*}
Consequently,
\begin{align*}
|f|_{q} & = \sup_{z\in\B}|f(z)|(1-|z|)^{q} \lesssim L = \sup_{k}\dfrac{|b_k|}{m_k^q}.
\end{align*}
This completes to proof of (b).
\end{proof}

\begin{cor}
If $0<p_1<p_2\leq n$, then we have
$$\mathcal{N}_{p_1}\subsetneq \mathcal{N}_{p_2}\subsetneq A^{-\frac{n+1}{2}}.$$
\end{cor}

\begin{proof}
Define
\begin{equation} \label{f1andf2}
f_1(z)=\sum_{k=0}^\infty 2^{\frac{k(n+1)}{2}} P_{2^k}(z), \quad f_2(z)=\sum_{k=0}^\infty 2^{\frac{k(1+p_1)}{2}} P_{2^k}(z)
\end{equation}
for $z\in\B$. Using Theorem \ref{T:HadamardGapClass}, it can be checked with a direct computation that $f_1 \in A^{-\frac{n+1}{2}} \backslash \calN_{p_2}$ and $f_2 \in  \calN_{p_2} \backslash \calN_{p_1}$.
\end{proof}


\bigskip

\end{document}